\documentclass[10pt]{amsart}
\usepackage{amsmath, amsfonts,amsthm,times,graphics}

\usepackage[active]{srcltx}
 \makeatletter
\renewcommand*\subjclass[2][2000]{%
  \def\@subjclass{#2}%
  \@ifundefined{subjclassname@#1}{%
    \ClassWarning{\@classname}{Unknown edition (#1) of Mathematics
      Subject Classification; using '1991'.}%
  }{%
    \@xp\let\@xp\subjclassname\csname subjclassname@#1\endcsname
  }%
}
 \makeatother
\newtheorem{theorem}{Theorem}[section]
\newtheorem{lemma}[theorem]{Lemma}

\theoremstyle{definition}

\newtheorem{remark}[theorem]{Remark}
\numberwithin{equation}{section}
 \makeatletter
\renewcommand*\subjclass[2][2000]{%
  \def\@subjclass{#2}%
  \@ifundefined{subjclassname@#1}{%
    \ClassWarning{\@classname}{Unknown edition (#1) of Mathematics
      Subject Classification; using '1991'.}%
  }{%
    \@xp\let\@xp\subjclassname\csname subjclassname@#1\endcsname
  }%
}

\def\NABLA#1{{\mathop{\nabla\kern-.5ex\lower1ex\hbox{$#1$}}}}
\def\Nabla#1{\nabla\kern-.5ex{}_{#1}}
\def\Tabla#1{\Tilde\nabla\kern-.5ex{}_{#1}}
\renewcommand{\Tilde}{\widetilde}

 \makeatother

\begin{document}

\title{Bi-Harmonic mappings and J. C. C. Nitsche type conjecture}
\author{David Kalaj}
\address{University of Montenegro, faculty of natural sciences and mathematics,
Cetinjski put b.b. 81000, Podgorica, Montenegro}
\email{davidkalaj@gmail.com}
\author{Saminathan Ponnusamy} \address{Department of
Mathematics, Indian Institute of Technology Madras, Chennai-600 036,
India.} \email{samy@iitm.ac.in}

%\dedicatory{This paper is dedicated to our authors.}

\begin{abstract}
In this note it is formulated the J. C. C. Nitsche type conjecture
for bi-harmonic mappings. The conjecture has been motivated by the
radial bi-harmonic mappings between annuli.
%\\
%\\
%{\sc R\'esum\'e.} Dans ce papier nous étendons Rado-Choquet-Kneser
%th\'{e}or\`{e}me pour les correspondances avec les données sur les
%limites Lipschitz et essentiellement positif jacobienne \`{a} la
%limite, sans restriction sur la convexit\'{e} du domaine de l'image.
%Il s'agit d'une extension d'un r\'{e}sultat r\'{e}cent de
%Alessandrini et Nesi \cite{ale}. Certaines demandes de la famille
%d'applications Quasiconformal harmonique entre la Jordanie sont des
%domaines donnés.

\end{abstract}
\maketitle %\tableofcontents
\section{Introduction}
The bi-harmonic equation of four times continuously differentiable
complex-valued functions $u$ defined in an open and connected set is
$$\Delta^2 u =\Delta (\Delta u)= 0.
$$
It is well known that, a planar harmonic mapping $u$ is bi-harmonic,
if and only if $u(z)=|z|^2g(z)+h(z)$, where $g$ and $h$ are harmonic
mappings, i.e. the mappings $w$ satisfying the Laplace equation
$\Delta w = 0$ in some subdomain $\Omega$ of the complex plane
${\mathbb C}$. Every analytic function is a harmonic mapping and
every bi-holomomorphic function is a harmonic diffeomorphism. The
set $A(1,t):=\{z: 1<|z|<t\}\subset {\mathbb C}$ is called an
annulus. It is well known the Schottky theorem which assert that two
annuli can be mapped by mean of a bi-holomorphic mapping if and only
if they have the same modulus.

J. C. C. Nitsche \cite{n} by considering the complex-valued
univalent harmonic functions
\begin{equation}\label{radial}
f(z)=\frac{ts-t^2}{(1 - t^2)}\frac 1{\bar z} +\frac{1 - t s}{1 -
t^2}z,
\end{equation}
showed that an annulus $1 < |z| < t$ can then be mapped onto any
annulus $1 < |w| < s$ with
\begin{equation}\label{nitsche}
s\geq n(t):=\frac{1+t^2} {2t}.
\end{equation}
J. C. C. Nitsche  conjectured that, condition \eqref{nitsche} is
necessary as well. The critical Nitsche map with zero initial speed
is
$$f(z)=\frac{1 + |z|^2}{2 \bar z}.
$$
This mean that this harmonic function make the maximal distortion of
rounded annuli $A(1,t)$.

Nitsche also showed that $s\geq s_0$ for some constant
$s_0=s_0(t)>1$. Thus although the annulus $1 < |z| < t$ can be
mapped harmonically onto a punctured disk, it cannot be mapped onto
any annulus that is ``too thin''. For the generalization of this
conjecture to $\Bbb R^n$ and some related results we refer to
\cite{jmaa}. For the case of hyperbolic harmonic mappings we refer
to \cite{h}. Some other generalization has been done in
\cite{israel}. The Nitsche conjecture for Euclidean harmonic
mappings is settled recently in \cite{conj} by Iwaniec, Kovalev and
Onninen, showing that, only radial harmonic mappings
$g(z)=e^{i\alpha} f(z)$, where $f$ is defined in \eqref{radial},
which inspired the Nitsche conjecture, make the extremal distortion
of rounded annuli. For some partial result toward the Nitsche
conjecture and some other generalizations we refer to the papers
\cite{ar, ar1, ar2, ako, iwa, dist, Al,weit}.
%\cite{ar}, \cite{ar1}, \cite{ar2}, \cite{ako}, \cite{iwa},
%\cite{dist},  \cite{Al} and \cite{weit}.

In this paper, we will state a similar conjecture with respect to
bi-harmonic mappings. In order to do this, in Section~2 we will find
all radial bi-harmonic maps between annuli. In Section~3 we will
prove some technical results concerning the radial bi-harmonic
mappings. Section~4 contains the main result which assert that, the
class of radial bi-harmonic diffeomorphisms between two annuli
$A(1,t)$ and $A(1,s)$ is nonempty if and only if $s\ge \sigma(t)$
where $\sigma(t)$ is some constant larger than $1$. It remains an
open question, whether this phenomenon remains true for the whole
class of bi-harmonic diffeomorphisms as in harmonic case.

\section{Radial solutions of bi-harmonic equation}
A mapping $f$ is called \emph{radial} if there exists a constant
$\varphi$ and a real function $g$ such that
$$f(re^{i\theta})=g(r)e^{i(\theta+\varphi)}.
$$
It is well known that, a radial solution $u$ of the harmonic
equation is given by $u(z) = A z + B/ \bar z$, where $a$ and $b$ are
two complex constants. To prove this we start by Laplacean in polar
coordinates. Let $U(r,\theta)= u(re^{i\theta})$. Then $\Delta u=0$
if and only if
$$  \Delta U := {1 \over r} {\partial \over \partial r} \left( r
{\partial U \over \partial r} \right) + {1 \over r^2} {\partial^2 U
\over \partial \theta^2}=0 .
$$

Assuming that $U(r,\theta)=p(r) e^{i\theta}$ we obtain the equation
$${1 \over r^2} (r^2p''(r)+rp'(r)- p(r))=0.
$$
By taking the change of variables $t=\log r$ and $P(t)=p(r)$ we
arrive at
$$P''(t)-P(t)=0
$$
and so, we obtain $P(t) = Ae^{t}+Be^{-t}$ and therefore $p(r) = A r+
B/r $. Thus
$$u(z)=A z+ \frac B{\bar z}.
$$
If $v$ is bi-harmonic, then the mapping $u=\Delta v$ is harmonic. If
$v$ is radial, then $u$ is radial as well. It  follows that
$$\Delta v= Az +\frac{B}{\bar z}
$$
for some real constants $A$ and $B$. Take $V(r,\theta)=v(r
e^{i\theta})$. Then we have
\begin{equation}\label{bih}
{1 \over r} {\partial \over \partial r} \left( r {\partial V \over
\partial r} \right) + {1 \over r^2} {\partial^2 V \over \partial
\theta^2}=\left(A r +\frac{B}{r}\right)e^{i\theta}.
\end{equation}
Now, we put
$$V(r,\theta)=g(r)e^{i\theta}.
$$
Then \eqref{bih} is equivalent with
$${1 \over r^2} (r^2g''(r)+rg'(r)- g(r))=A r +\frac{B}{r}.
$$
By taking again the change of variables $t=\log r$, $G(t)=g(r)$ we
arrive at the equation
$$G''(t)-G(t)= Ae^{3 t} +B e^t.
$$
Thus
$$G(t) = de^{-t}+a e^t + b t e^t +c e^{3t}, \ \  a,b,c,d\in{\mathbb R}
$$
and therefore
\begin{equation}\label{gq}
g(r) = \frac dr + a r + br \log r + c r^3.
\end{equation}
It follows that, every radial solution of bi-harmonic equation has
the form:
$$f(z)=\frac{d}{\bar z}+ a z + bz \log|z| + c |z|^2 z.
$$

\section{The technical lemmas}
\begin{remark}
Throughout the paper we will assume that the bi-harmonic mapping
$f(re^{it})=g(r)e^{it}:A(1,t)\to A(1,s)$ maps the inner boundary
onto the inner boundary of corresponding annulus, i.e. $g$ defined
in \eqref{gq} is increasing. A similar analysis works for the case
when $g$ is a decreasing function. A radial harmonic mapping $f$,
will be called \emph{homogeneous} if the initial and final speeds
are equal to zero, i.e. if $g'(1)=g'(t)=0$
\end{remark}
\begin{lemma}\label{abuv}
Assume that $t>1$ and $s>1$. If the function $g$ defined by
\eqref{gq} satisfies $g(1)=1$, $g(t)=s$, $g'(1)=x>0$, $g'(t)=y>0$,
then there exists a function $h$, such that $h(1)=1$, $h(t)=s$,
$h'(1)=0$, $h'(t)=0$,
$$g(r)=A(r)+B(r) s +  U(r)x+V(r)y,
$$
and
$$h(r)=A(r)+B(r) s
$$
where
\begin{eqnarray*}
A &=&\displaystyle  \frac{(3 + r^2) (r^2 - t^2) ( t^2-1) +
  2 r^2 (3 - 2 t^2 - t^4) \log r +
  2 (r^4 + t^4 + r^2 (-3 + t^4)) \log t}{
 4 r (-1 + t^2) (1 - t^2 + (1 + t^2) \log t)},\\
B &=&\displaystyle \frac{2 r^2 (-1 - 2 t^2 + 3 t^4) \log r - (-1 +
r^2) ((-1 + t^2) (r^2 + 3 t^2) +
     2 (-1 + r^2) t^2 \log t)}{
 4 r t (-1 + t^2) (1 - t^2 + (1 + t^2) \log t)},\\
U &=&\displaystyle \frac{-2 r^2 (-1 + t^2)^2 \log r + (-1 + r^2)
((r^2 - t^2) (-1 + t^2) - 2 (r^2 - t^4)
    \log t)}{4 r (-1 + t^2) (1 - t^2 + (1 + t^2) \log t)},
\end{eqnarray*}
and
$$V = \frac{-2 r^2 (-1 + t^2)^2 \log r + (-1 + r^2) (-(r^2 - t^2) (-1 + t^2) +
     2 (-1 + r^2) t^2 \log t)}{
 4 r (-1 + t^2) (1 - t^2 + (1 + t^2) \log t)}.
 $$
\end{lemma}
\begin{proof}
The proof is length but straightforward and is therefore omitted.
\end{proof}

\begin{lemma}\label{lem-3.3}
Under the conditions and notation of Lemma {\rm \ref{abuv}} there
hold the following relations
\begin{equation}\label{bplus}
B'(r)>0, \ \ \ \text{for}\ \  \ 1<r<t<\infty,
\end{equation}

\begin{equation}\label{aplus}
-A'(r)>0, \ \ \ \text{for}\ \  \ 1<r<t<\infty,
\end{equation}

\begin{equation}\label{apa}
A'(t^-)=B'(t^-)=0, \ \ \ \text{for}\ \  \ 1<t<\infty,
\end{equation}

\begin{equation}\label{us}
-\frac{d}{dr}\frac{A'(r)}{B'(r)}>0, \ \ \ \text{for}\ \  \
1<r<t<\infty,
\end{equation}

\begin{equation}\label{lus}
-\frac{d}{dr}\frac{U'(r)}{B'(r)}>0, \ \ \ \text{for}\ \  \
1<r<t<\infty,
\end{equation}

\begin{equation}\label{luso}
\frac{d}{dr}\frac{V'(r)}{B'(r)}>0, \ \ \ \text{for}\ \  \
1<r<t<\infty,
\end{equation}
\begin{equation}\label{do}
\lim_{r\to t^-}\frac{-A'(r)}{B'(r)}=\frac{t (3 - 4 t^2 + t^4 + 4 t^2
\log t)}{2 - 2 t^2 + \log t +
 3 t^4\log t},
 \end{equation}

 \begin{equation}\label{re}
 \lim_{r\to t^-}\frac{-U'(r)}{B'(r)}=\frac{t (-1 + t^4 - 4 t^2 \log t)}{2 - 2 t^2 + \log t + 3 t^4
 \log t},\end{equation}

\begin{equation}\label{mi}
\lim_{r\to t^-}\frac{-V'(r)}{B'(r)}=-\infty,
\end{equation}

\begin{equation}\label{fa}
\lim_{r\to 1^+}\frac{-A'(r)}{B'(r)}=-\frac{t (-2 t^2 (-1 + t^2) + (3
+ t^4) \log t)}{ 1 - 4 t^2 + 3 t^4 - 4 t^2 \log t},
\end{equation}

\begin{equation}\label{so}
\lim_{r\to 1^+}\frac{-U'(r)}{B'(r)}=-\infty
\end{equation}
and
\begin{equation}\label{la}
\lim_{r\to 1^+}\frac{-V'(r)}{B'(r)}=\frac{t (-1 + t^4 - 4 t^2 \log
t)}{1 - 4 t^2 + 3 t^4 - 4 t^2 \log t}
\end{equation}
\end{lemma}
The proof of Lemma~\ref{abuv} lies on the following lemma.
\begin{lemma}\label{1}
For all $1< t$, $1< r< t$,
\begin{enumerate}
\item[\textbf{(a)}]
$2 r^2 (-1 - 2 t^2 + 3 t^4) \log r +(1 - r^2) (3 (r^2 - t^2) (-1 +
t^2) +
    2 (1 + 3 r^2) t^2 \log t)>0,$

\item[\textbf{(b)}] $-2 r^2 (-3 + 2 t^2 + t^4) \log r + (r^2-1) (3 (r^2 - t^2) (-1 +
t^2) + 2 (3 r^2 + t^4)
   \log t)<0,$

\item[\textbf{(c)}] $(-1 +
     t^2) (3 (-1 + r^2) (-r^2 + t^2) + 2 r^2 (1 + 3 t^2) \log r) +
  2 (1 + 2 r^2 - 3 r^4) t^2 \log t
> 0,$

\item[\textbf{(d)}] $2 (r^4 - t^2) (t^2-1 ) \log r + (1 - r^2) ((r^2 - t^2) (-1 +
t^2) +      2 (r^2-1) t^2 \log t)> 0$
\item[\textbf{(e)}] $1 - t^2 + (1 + t^2) \log t> 0.$
\end{enumerate}
\end{lemma}

\begin{proof}[Proof of Lemma~{\rm \ref{1}}]
By taking the substitution $\alpha=r^2$, $\beta=t^2$, the inequality
\textbf{(d)} o f the lemma is equivalent with the inequality
%\begin{equation}\label{uno} %not used
$$h(\beta):=(\alpha^2 - \beta) (-1 + \beta) \log
\alpha + (\alpha - 1) ((\beta - \alpha) (\beta - 1) + (1 - \alpha)
\beta \log \beta)\ge 0.
$$
A computation gives
\begin{eqnarray*}
h'(\beta)& =&(1 + \alpha^2 - 2 \beta) \log \alpha - (-1 + \alpha) (2
(\alpha - \beta) + (-1 + \alpha) \log \beta),\\
h''(\beta)& =& -\frac{(-1 + \alpha) (-1 + \alpha - 2 \beta)}{\beta}
- 2 \log \alpha
\end{eqnarray*}
and
$$h'''(\beta) =\frac{(-1 + \alpha)^2}{\beta^2}.
$$
It follows that $h''$ is increasing, and therefore
$$h''(\beta)\ge h''(\alpha)=-1/\alpha + \alpha - 2 \log \alpha.
$$
But
$$(-1/\alpha + \alpha - 2 \log \alpha)'=\frac{(\alpha-1)^2}{\alpha^2},
$$
and therefore $-1/\alpha + \alpha - 2 \log \alpha\ge -1+1-2 \log
1=0$. It follows that
%\begin{equation}\label{due}%not used
$$h''(\beta)\ge 0.
$$
Thus, we have
%\begin{equation}\label{tre} %not used
$$h'(\beta)\ge h'(\alpha)=0.
$$

It follows finally that
$$h(\alpha)\ge h(\beta)=0.
$$
 The proofs of \textbf{(a)}, \textbf{(b)} and
\textbf{(c)} and \textbf{(e)} are similar to the proof of
\textbf{(d)} and are therefore omitted.
\end{proof}

\begin{proof}[Proof of Lemma~{\rm \ref{lem-3.3}}]
First of all
$$A'(r)=\frac{2 r^2 (3 - 2 t^2 - t^4) \log r + (r^2-1) (3 (r^2 - t^2) (-1 + t^2) +    2 (3 r^2 + t^4) \log t)}{4 r^2 (-1 + t^2) (1 -
   t^2 + (1 + t^2) \log t)},
$$
$$B'(r)=\frac{2 r^2 (3 t^4 - 2 t^2 -1 ) \log r  + (1 - r^2) (3 (r^2 - t^2) (-1 + t^2) +
    2 (1 + 3 r^2) t^2 \log
t)}{4 r^2 t (-1 + t^2) (1 -
   t^2 + (1 + t^2) \log t)},
$$
$$U'(r) = \frac{(1 + 3 r^2) (r^2 - t^2) (t^2-1) +
  2 r^2 (1 - t^2)^2 \log r  -
  2 (3 r^4 - t^4 - r^2 (1 + t^4)) \log
t}{
 4 r^2 (-1 + t^2) (1 - t^2 + (1 + t^2) \log t )},
$$
and
$$ V'(r) = \frac{2 r^2 (1 - t^2)^2 \log r  + (1 - r^2) ((-1 +
t^2) (3 r^2 + t^2) +
     2 (1 + 3 r^2) t^2 \log
t)}{
 4 r^2 (-1 + t^2) (1 - t^2 + (1 + t^2) \log t)}.
$$

Lemma~\ref{1}\textbf{(a)} and \textbf{(b)} imply that $A'(r)<0$ and
$B'(r)>0$. This proves the inequalities \eqref{bplus} and
\eqref{aplus}. The relation \eqref{apa} follows at once.

The derivative of the quotient function $-A'(r)/B'(r)$ is
\[\begin{split}&12 r t (-1 + t^2) (1 -
   t^2 + (1 + t^2) \log t) \\ &\times\frac{2 (r^4 - t^2) (-1 + t^2) \log r - (-1 + r^2) ((r^2 - t^2)
(-1 + t^2) +
      2 (-1 + r^2) t^2 \log t)}{(2 r^2 (1 + 2 t^2 - 3 t^4) \log r + (-1 + r^2) (3 (r^2 - t^2) (-1 + t^2) +
     2 (1 + 3 r^2) t^2 \log t))^2}.\end{split}\]
Thus, Lemma~\ref{1}\textbf{(d)} and \textbf{(e)} imply that the last
expression is positive.
Thus \eqref{us} is proved. The proofs of %\eqref{aplus}, \eqref{bplus},
\eqref{lus} and \eqref{luso} are similar. The proofs of relations
\eqref{do}-\eqref{la} are similar to each other and follow by
l'Hopital's rule. See Figure~1 for the geometric interpretation of
\eqref{so} and \eqref{la}.
\end{proof}

\begin{figure}[htp]\label{poi}
\centering
\includegraphics{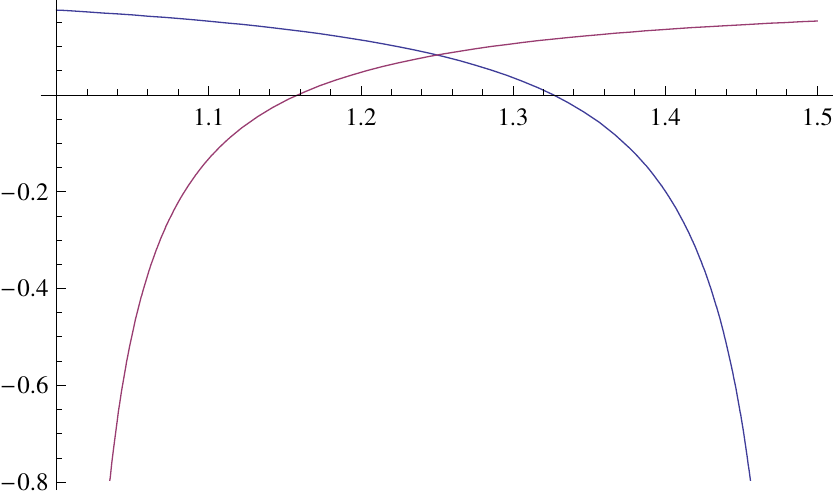}
\caption{These two curves are graphs of the functions $-U'(r)/B'(r)$
and $-V'(r)/B'(r)$ for $t=3/2$, and $1<r<t$.}
\end{figure}

\begin{lemma}\label{xyz}
For  every $t>1$, and $\tau =\frac{1+t}{2}$, we have
$$\frac{-A'(\tau)}{B'(\tau)}>1.
$$
\end{lemma}
\begin{proof}
Namely
$$\frac{-A'(\tau)}{B'(\tau)}>1
$$
if and only if
$$\varphi(t):=(1 - t^2) (9 + 30 t + 9 t^2 + 8 (1 + t)^2 \log \tau)-2 t (9 + 18 t + 17 t^2 + 4 t^3) \log t >0.
$$
On the other hand
$$\varphi^{(5)}(t)=\frac{12 (9 + 6 t + 2 t^2 + 6
t^3 + 9 t^4)}{t^4 (1 + t)^2}>0
$$
and  $\varphi^{(k)}(1)= 0$, for $k=0,1,2,3,4$. We therefore deduce
the following sequence of inequalities $\varphi^{(4)}(t)> 0$,
$\varphi^{(3)}(t)> 0$, $\varphi''(t)> 0$, $\varphi'(t)> 0$ and
$\varphi(t)>0$ for $t>1$.
\end{proof}

\begin{lemma}
Under the conditions and notation of Lemma~{\rm \ref{abuv}} we have
$U'(r)=V'(r)$ if and only if
$$r=\rho:= \sqrt{\frac 16 + \frac {t^2}{6} + \frac 16 \sqrt{1 + 14 t^2 + t^4}}.
$$
Moreover
\begin{equation}\label{uvb}
-U'(\rho)=-V'(\rho)>0,\ \ \ \frac{-A'(\rho)}{B'(\rho)}>1.
\end{equation}
\end{lemma}
\begin{proof}
As
$$U'(r)-V'(r)=\frac{-3 r^4 + t^2 + r^2 (1 + t^2)}{2 r^2 (-1 + t^2)},
$$
it follows that
$$U'(r)=V'(r)~\text {  if and only if }~ r=\rho:= \sqrt{\frac 16 + \frac {t^2}{6} + \frac 16 \sqrt{1 + 14 t^2 +
t^4}}
$$ or what is the same
$$t = \frac{\rho \sqrt{-1 + 3 \rho^2}}{\sqrt{1 + \rho^2}}.
$$ By taking the substitution $\kappa=\rho^2$ and $\eta=t^2$, we obtain
\[\begin{split}-U'(\rho)&=\frac{(1 + 3 \rho^2) (\rho^2 - t^2) (t^2-1 )
-
 2 \rho^2 (1 - t^2)^2 \log \rho - (6 \rho^4 - 2 t^4 - 2 \rho^2 (1 + t^4))
 \log
   t }{4 \rho^2 (-1 + t^2) (1 - t^2 + (1 + t^2) \log t)}\\&=\frac{(1 + 3 \kappa) (\kappa - s) (-1 + \eta) -
 \kappa (-1 + \eta)^2 \log \kappa + (\kappa - 3 \kappa^2 + \eta^2 + \kappa \eta^2) \log \eta}{(2 \kappa (-1 +
   \eta) (2 - 2 \eta + (1 + \eta) \log \eta)}.\end{split}\]
Since
$$\eta=\frac{\kappa(3\kappa-1)}{1+\kappa} ~\mbox{ and }~ { (2 - 2\eta + (1 + \eta) \log \eta)}>0 ~\mbox{ for }~ \eta>1,
$$
we have to prove that
\begin{equation*}\label{ara}
L(\kappa):=(1 + 3 \kappa) (\kappa - s) (-1 + \eta) -
 \kappa (-1 + \eta)^2 \log \kappa + (\kappa - 3 \kappa^2 + \eta^2 + \kappa \eta^2) \log \eta>0.
\end{equation*}
Then $$L(\kappa)=\frac{\kappa (-1 - 2 \kappa + 3
\kappa^2)}{(1+\kappa)^2}K(\kappa)$$ where
$$K(\kappa)=
-2 - 4 \kappa +
   6 \kappa^2 + (-1 - 2 \kappa + 3 \kappa^2) \log \kappa + (1 - 2 \kappa - 3 \kappa^2) \log\left(
     \frac{\kappa (-1 + 3 \kappa)}{1 + \kappa}\right).
$$
Further
$$K'''(\kappa)=\frac{4 (1 - 4 \kappa + 14 \kappa^2 + 12
\kappa^3 + 9 \kappa^4)}{\kappa^2 (-1 + 2 \kappa + 3 \kappa^2)^2}>0.
$$
Moreover
$$K''(1)= 0,\ \ \ K'(1)= 0,\ \ \ K(1)= 0
$$
and therefore
$$K(\kappa)>0.
$$
Since $r\to -A'(r)/B'(r)$ is increasing and
$$\rho=\sqrt{\frac 16 + \frac{t^2}{6} + \frac{1}{6} \sqrt{1 + 14 t^2 + t^4}}
> \tau=\frac{1 + t}{2},
$$
by Lemma~\ref{xyz} and \eqref{us}, we obtain
%\begin{equation}\label{aba} %not used
$$\frac{-A'(\rho)}{B'(\rho)}>\frac{-A'(\tau)}{B'(\tau)}>1.
$$
\end{proof}
\section{The main results}
As a direct corollary of Lemma~\ref{abuv} we obtain
\begin{theorem}
If $f(re^{i\theta})=h(r)e^{i\theta}$, $h(1)=1$, $h(t)=s$, $h'(1)=0$,
$h'(t)=0$,  is a radial homogeneous  bi-Harmonic diffeomorphism of
the annulus $A(1,t)$ onto the annulus $A(1,s)$, then
%\begin{equation}\label{advice} %not used
$$
s\ge \sigma_0(t):=\frac{t (3 - 4 t^2 + t^4 + 4 t^2 \log t)}{2 - 2
t^2 + \log t + 3 t^4 \log t}.
$$ The condition is sufficient as well.
The critical homogeneous bi-harmonic mapping is
$$f(z)=h_0(r)e^{i\theta},\ \  z=r e^{i\theta}
$$
where
\[\begin{split}h_0(r)&=\frac{(1-t^2) (3 t^2 + 3 (3 - t^2) r^2 - r^4)}{{4 r (2 - 2 t^2 + \log t + 3 t^4
\log t)}} \\
& \hspace{.5cm} + \frac{(6 t^4 + 6 (1 + t^4) r^2 - 2 r^4) \log t + 6
(1 - t^2)^2 r^2 \log r}{{4 r (2 - 2 t^2 + \log t + 3 t^4 \log t)}}.
\end{split}\]
 The function $\sigma_0(t)$ is
smaller than the corresponding function $n(t)$ for harmonic
mappings. See Figure~{\rm 2}.
\end{theorem}

\begin{theorem}[The main theorem]
Let $t>1$ and $s>1$. If $f(re^{i\theta})=g(r)e^{i\theta}$ is a
radial bi-Harmonic diffeomorphism of the annulus $A(1,t)$ onto the
annulus $A(1,s)$, mapping the inner boundary onto the inner
boundary, then $s\ge \sigma(t)$ where the constant
$$\sigma(t)=\inf_{x\ge 0,y\ge 0}\sup_{1\le r \le t}\left\{\frac{-A'(r)}{B'(r)}+x
\frac{-U'(r)}{B'(r)}+y\frac{-V'(r)}{B'(r)}\right\}
$$
is bigger than $1$ and smaller than $\sigma_0(t)$. The condition is
sufficient as well. Moreover there exists a critical mapping
$f(re^{it})=g_0(r)e^{it}$, between annuli $A(1,t)$ and
$A(1,\sigma(t))$ and it satisfies the conditions $g'_0(1)>0$ and
$g'_0(t)>0$.
\end{theorem}
\begin{proof} Under the conditions of the theorem $g$ is non-decreasing.
 Then $$g'(r)\ge 0, \ \ \text{ for }  \ \ 1\le r\le t,$$ if and only if $$A'(r)+B'(r)s + U'(r)x+ V'(r) y\ge
 0\text{ for }  \ \ 1\le r\le t.
$$
Here $x=g'(1)$ and $y=g'(t)$. If
$$X_n(t) :=\frac{-A'(r_n)}{B'(r_n)}+x_n \frac{-U'(r_n)}{B'(r_n)}+y_n\frac{-V'(r_n)}{B'(r_n)}\to \sigma(t)
$$
then by \eqref{uvb}
$$X_n\ge \frac{-A'(\rho)}{B'(\rho)}+x_n \frac{-U'(\rho)}{B'(\rho)}+y_n\frac{-V'(\rho)}{B'(\rho)}>
\frac{-A'(\rho)}{B'(\rho)}>1.
$$
It follows that $\sigma(t)>1$ and that the sequences $x_n$ and $y_n$
stay bounded when $n\to \infty$. On the other hand, it follows from
\eqref{do}-\eqref{so}
%\eqref{do},\eqref{re}, \eqref{mi}, \eqref{fa}, \eqref{so}
that there exist $1<\tau_1(t)<\tau_2(t)<t$ such that the function
$$p(r)=-\frac{U'(r)+V'(r)}{B'(r)}
$$
is negative in intervals $[1,\tau_1]$ and $[\tau_2,t]$. From
\eqref{us}, the maximum of $\frac{-A'(r)}{B'(r)}$ is
$$\frac{-A'(t)}{B'(t)}:=\frac{-A'}{B'}(t^-)
$$
defined in \eqref{do}.  Thus there exists a small enough $x>0$ such
that
$$\frac{-A'(t)}{B'(t)}> \frac{-A'(r)}{B'(r)}+xp(r)
$$
for all $r:1<r<t$ and fixed $t$. This means that
$$\sigma(t)<\sigma_0(t).
$$
Assume without loos of generality that $x_n \to x_0$ and $y_n \to
y_0$ and $r_n\to r_0$. The sequence $g_n$ is monotonic and converges
to a strictly monotonic function $g_0.$ The resulting bi-harmonic
mapping is critical. Since $\sigma<\sigma_0$, because
$A'(t^-)=B'(t^-)=0$, $U'(t^-)=0$, $V'(t^-)=1$ and $(-U'/B')(t^-)>0$
it follows that $x_0>0$, $y_0>0$ and $r_0<t$.
\end{proof}

\begin{figure}[htp]\label{poincare1}
\centering
\includegraphics{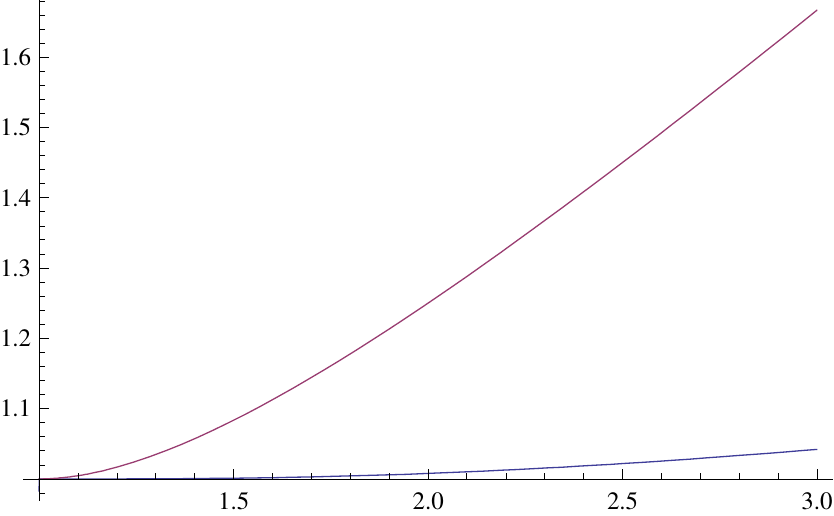}
\caption{The curve above (below) corresponds to the critical
harmonic (bi-harmonic) mappings between annuli $A(1,t)$, and
$A(1,\omega(t))$, where $\omega(t)=\sigma_0(t)$ and $\omega(t)=n(t)$
respectively ($1<t\le 3$).}
\end{figure}

\begin{remark}
\begin{enumerate}
\item[\textbf{(a)}] For given $t>1$ do there exists some
constant $s(t)>1$ such that, the class of bi-harmonic
diffeomorphisms between annuli $A(1,t)$ and $A(1,s(t))$ is empty?
\item[\textbf{(b)}] If the answer to the question a) is affirmative, does
$s(t)=\sigma(t)$? In \cite{jmaa} the first author proved that, if
$\Delta u$ is small enough, then the answer of question \textbf{(a)}
is affirmative.
\end{enumerate}
\end{remark}

\end{document}